\documentclass[12pt,reqno,a4wide]{amsart}
\usepackage{amsfonts}

\theoremstyle{plain}
\newtheorem{theorem}{Theorem}

\newtheorem{lemma}{Lemma}

\theoremstyle{definition}
\newtheorem{definition}{Definition}

\allowdisplaybreaks

\newcommand{\bell}{\textup{B}}

\begin{document}
	
\title{A formula for the $r$-coloured partition function in terms of the sum of divisors function and its inverse}

\author{Sumit Kumar Jha}
\address{International Institute of Information Technology\\
Hyderabad, India}
\email{kumarjha.sumit@research.iiit.ac.in}



\begin{abstract}
Let $p_{-r}(n)$ denote the $r$-coloured partition function, and $\sigma(n)=\sum_{d|n}d$ denote the sum of positive divisors of $n$. The aim of this note is to prove the following
$$
p_{-r}(n)=\theta(n)+\,\sum_{k=1}^{n-1}\frac{r^{k+1}}{(k+1)!}  \sum_{\alpha_1\,=  k}^{n-1} \, \sum_{\alpha_2\,=  k-1}^{\alpha_1-1} \cdots  \sum_{\alpha_k\, = 1}^{\alpha_{k-1}-1}\theta(n-\alpha_1) \theta(\alpha_1 -\alpha_2) \cdots \theta(\alpha_{k-1}-\alpha_k) \theta(\alpha_k)
$$
where $\theta(n)=n^{-1}\, \sigma(n)$, and its inverse
$$
\sigma(n)=n\,\sum_{r=1}^{n}\frac{(-1)^{r-1}}{r}\,\binom{n}{r}\,p_{-r}(n).
$$
\end{abstract}

\maketitle

\section{Main results}
\begin{definition}\cite{Chern}
A partition of a positive integer $n$ is a finite weakly decreasing sequence of positive integers $\lambda_{1}\geq \lambda_{2}\geq \cdots \geq \lambda_{r}>0$ such that $\sum_{i=1}^{r}\lambda_{i}=n$. The $\lambda_{i}$'s are called the \emph{parts} of the partition. Let $p(n)$ denote the number of partitions of $n$.\par 
A partition is said to be $r$-coloured if each part can occur as $r$ colours. Let $p_{-r}(n)$ denote the number of $r$ coloured partitions of $n$. The generating function for $p_{-r}(n)$ is given by
$$
\sum_{n=0}^{\infty}p_{-r}(n)q^{n}=E(q)^{-r}
$$
where
$$E(q):=\prod_{j=1}^{\infty}(1-q^{j})$$
where $|q|<1$.
\end{definition}
In the following, let $\sigma(n)=\sum_{d|n}d$ denote the sum of divisors of a positive integer $n$. \par 
We first prove the following.
\begin{theorem}
For all positive integers $n\geq 2$ we have
\begin{equation}
\label{th1}
p_{-r}(n)=\theta(n)+\sum_{k=1}^{n-1}\frac{r^{k+1}}{(k+1)!} \sum_{\alpha_1\,=  k}^{n-1} \, \sum_{\alpha_2\,=  k-1}^{\alpha_1-1} \cdots  \sum_{\alpha_k\, = 1}^{\alpha_{k-1}-1}\theta(n-\alpha_1) \theta(\alpha_1 -\alpha_2) \cdots \theta(\alpha_{k-1}-\alpha_k) \theta(\alpha_k)
\end{equation}
where $\theta(n)=n^{-1}\, \sigma(n)$.
\end{theorem}
\begin{lemma}
We have 
$$
-\log(E(q))=\sum_{n=1}^{\infty}\theta(n)\,q^{n}.
$$
\end{lemma}
\begin{proof}
It is easy to see that
\begin{align*}
\log(E(q))&=\sum_{j=1}^{\infty}\log(1-q^{j})\\
&=-\sum_{j=1}^{\infty}\sum_{l=1}^{\infty}\frac{q^{lj}}{l}\\
&=-\sum_{n=1}^{\infty}q^{n}\left(\sum_{d|n}\frac{1}{d}\right).
\end{align*}
This completes the proof.
\end{proof}
\begin{definition}
For $n$ and $k$ non-negative integers, the partial $(n,k)$th partial Bell polynomials in the variables $x_{1},x_{2},\dotsc,x_{n-k+1}$ denoted by $B_{n,k}\equiv\bell_{n,k}(x_1,x_2,\dotsc,x_{n-k+1})$ \cite[p. 206]{Comtet} can be defined by
\begin{equation*}
\bell_{n,k}(x_1,x_2,\dotsc,x_{n-k+1})=\sum_{\substack{1\le i\le n,\ell_i\in\mathbb{N}\\ \sum_{i=1}^ni\ell_i=n\\ \sum_{i=1}^n\ell_i=k}}\frac{n!}{\prod_{i=1}^{n-k+1}\ell_i!} \prod_{i=1}^{n-k+1}\Bigl(\frac{x_i}{i!}\Bigr)^{\ell_i}.
\end{equation*}
\end{definition}
Cvijovi\'{c} \cite{Bell} gives the following formula for calculating these polynomials
\begin{align}
\label{explicit}
 B_{n, k + 1}  =  & \frac{1}{(k+1)!} \underbrace{\sum_{\alpha_1\,=  k}^{n-1} \, \sum_{\alpha_2\,=  k-1}^{\alpha_1-1} \cdots  \sum_{\alpha_k\, = 1}^{\alpha_{k-1}-1} }_{k }
\overbrace{\binom{n}{\alpha_1}  \binom{\alpha_1}{\alpha_2} \cdots  \binom{\alpha_{k-1}}{\alpha_k}}^{k}\nonumber
\\
&\cdot x_{n-\alpha_1} x_{\alpha_1 -\alpha_2} \cdots x_{\alpha_{k-1}-\alpha_k} x_{\alpha_k} \qquad(n\geq k+1, k\,=1, 2, \ldots)
\end{align}
\begin{lemma}
We have
\begin{equation}
    \label{lem2}
    p_{-r}(n)=\frac{1}{n!}\sum_{k=1}^{n}r^{k}\, B_{n,k}(1!\,\theta(1),2!\,\theta(2),\cdots,(n-k+1)!\, \theta(n-k+1)),
\end{equation}
where $\theta(n)=n^{-1}\sigma(n)$.
\end{lemma}
\begin{proof}
Let $f(q)=e^{rq}$, and $g(q)=-\log(E(q))$. Using Fa\`{a} di Bruno's formula \cite[p. 134]{Comtet} we have
\begin{equation}
\label{faa}
{d^n \over dq^n} f(g(q)) = \sum_{k=1}^n f^{(k)}(g(q))\cdot B_{n,k}\left(g'(q),g''(q),\dots,g^{(n-k+1)}(q)\right).
\end{equation}
Since $f^{(k)}(q)=r^{k}\,e^{rq}$ and $g(0)=1$, letting $q\rightarrow 0$ in the above equation gives us equation \eqref{lem2}.
\end{proof}
Combining equations \eqref{lem2} and \eqref{explicit} we can conclude \eqref{th1}.\par 
Now we prove the following.
\begin{theorem}
\label{main2}
We have
\begin{equation}
    \label{seceq}
    \sigma(n)=n\,\sum_{r=1}^{n}\frac{(-1)^{r-1}}{r}\,\binom{n}{r}\,p_{-r}(n).
\end{equation}
\end{theorem}
\begin{lemma}
We have
\begin{equation}
    \label{lem3}
    \theta(n)=\frac{1}{n!}\sum_{k=1}^{n}(-1)^{k-1}\, (k-1)!\, B_{n,k}(1!\, p(1),2!\, p(2),\cdots,(n-k+1)!\, p(n-k+1))
\end{equation}
where $p(n)=p_{-1}(n)$ is the partition function.
\end{lemma}
\begin{proof}
Let $f(q)=\log{q}$, and $g(q)=1/E(q)$. Then using Fa\`{a} di Bruno's formula \eqref{faa} we have
$$
{d^n \over dq^n} f(g(q)) = \sum_{k=1}^n f^{(k)}(g(q))\cdot B_{n,k}\left(g'(q),g''(q),\dots,g^{(n-k+1)}(q)\right).
$$
Since $f^{(k)}(q)=\frac{(-1)^{k-1}\, (k-1)!}{q^{k}}$ and $g(0)=1$, letting $q\rightarrow 0$ in the above equation gives us equation \eqref{lem3}.
\end{proof}
\begin{lemma}
We have for positive integers $n,k$
\begin{equation}
\label{lem4}
B_{n,k}(1!\, p(1),2!\, p(2),\cdots,(n-k+1)!\, p(n-k+1))=\frac{n!}{k!}\sum_{r=1}^{k}(-1)^{k-r}\binom{k}{r}p_{-r}(n)
\end{equation}
\end{lemma}
\begin{proof}
We start with the generating function for the partial Bell polynomials as follows
\begin{align*}
{\displaystyle \sum _{n=k}^{\infty }B_{n,k}(1!\, p(1),2!\, p(2),\cdots,(n-k+1)!\, p(n-k+1)){\frac {q^{n}}{n!}}}
&= {\frac {1}{k!}}\left(\sum _{j=1}^{\infty }p(j)q^{j}\right)^{k} \\
&=\frac{1}{k!}(E(q)^{-1}-1)^{k}\\
&=\frac{1}{k!}\sum_{r=0}^{k}(-1)^{k-r}\binom{k}{r}E(q)^{-r}\\
&=\frac{1}{k!}\sum_{r=0}^{k}(-1)^{k-r}\binom{k}{r}\sum_{n=0}^{\infty}p_{-r}(n)q^{n}
\end{align*}
to conclude the equation \eqref{lem4}.
\end{proof}
\begin{proof}[Proof of Theorem \ref{main2}]
Combining equations \eqref{lem3} and \eqref{lem4} we have
\begin{align*}
\sum_{d|n}\frac{1}{d}&=\sum_{k=1}^{n}\frac{1}{k}\sum_{r=1}^{k}(-1)^{r-1}\, \binom{k}{r}\,p_{-r}(n)\\
&=\sum_{r=1}^{n}(-1)^{r-1}p_{-r}(n)\sum_{k=r}^{n}\frac{1}{k}\binom{k}{r}\\
&=\sum_{r=1}^{n}\frac{(-1)^{r-1}}{r}\,\binom{n}{r}\,p_{-r}(n).
\end{align*}
Now we can conclude our main result equation \eqref{seceq} after the fact that $\sum_{d|n}\frac{n}{d}=\sum_{j|n}j$.
\end{proof}


\begin{thebibliography}{99}
\bibitem{Chern}
S. Chern, S. Fu \& D. Tang, Some inequalities for k-colored partition functions, \emph{Ramanujan J} 46, 713–725 (2018)
\bibitem{Bell}
D. Cvijovi\'{c}, New identities for the partial Bell polynomials, Applied mathematics letters, \textbf{24} (2011), 1544--1547.
\bibitem{Comtet} L. Comtet, {\em Advanced Combinatorics: The Art of Finite and Infinite Expansions}, D. Reidel Publishing Co.,  Dordrecht, 1974.
\end{thebibliography}
\end{document}